\documentclass[11pt]{article}

\usepackage{amsmath,amssymb,amsthm}
\usepackage{graphicx}
\usepackage{tikz}
\usetikzlibrary{arrows.meta,calc,positioning,shapes.geometric}
\usepackage{hyperref}
\hypersetup{
  colorlinks=false,
  linkbordercolor={1 0 0},
  citebordercolor={0 1 0},
  urlbordercolor={0 0 1},
  pdfborder={0 0 1},
  linktoc=all,
  bookmarksnumbered=true,
  bookmarksopen=true,
  pdfstartview=FitH,
  pdftitle={Counterexamples to Two Conjectures on Modular Edge Colorings of Graphs},
  pdfauthor={}
}

\allowdisplaybreaks

\newtheorem{theorem}{Theorem}[section]
\newtheorem{lemma}[theorem]{Lemma}
\newtheorem{corollary}[theorem]{Corollary}
\newtheorem{problem}[theorem]{Problem}
\newtheorem{conjecture}[theorem]{Conjecture}

\newcommand{\chik}{\chi_k'}

\title{Counterexamples to two conjectures on modular edge colorings of graphs}
\author{Chunqiang Guo
\thanks{The work was supported by the Excellent Graduate Student Innovation Project of Xinjiang University (XJDX2025YJS034).},
Baoyindureng Wu\footnote{Corresponding author. Email: baoywu@163.com.}
\thanks{The work was supported by the Open Project of the Key Laboratory in
Xinjiang Uygur Autonomous Region of China (2023D04026) and the National
Natural Science Foundation of China (12061073).}\\
\small College of Mathematics and System Sciences, Xinjiang University\\
\small Urumqi, Xinjiang 830046, P. R. China\\}
\date{}

\begin{document}

\maketitle

\begin{abstract}
For an integer $k\geq2$, let $\chi_k'(G)$ denote the minimum number of
colors in an edge-coloring of a graph $G$ such that every nonzero degree
in each color subgraph is congruent to $1\pmod{k}$. A graph is a
$0_k$-graph if every vertex degree is divisible by $k$. We disprove a
conjecture of Berthe et al.\ (On modular edge colorings of graphs, SIAM
J. Discrete Math. 40 (2026) 897--904), which states that
$\chi_k'(G)\leq k+o(k)$ for every $0_k$-graph $G$. We prove a lower
bound for $0_k$-graphs with degree set $\{k,2k\}$ and a specified vertex
partition. With a suitable choice of the part sizes, if the number of
edges inside one part is $o(k^2)$, then
$\chi_k'(G)\geq(4-2\sqrt2+o(1))k$. This gives connected bipartite and
connected nonbipartite counterexamples. In particular, the same examples
also disprove the earlier conjecture of Botler, Colucci, and Kohayakawa
(The mod $k$ chromatic index of graphs is $O(k)$, J. Graph Theory 102
(2023) 197--200), which states that $\chi_k'(G)\leq k+C$ for some
absolute constant $C$.
\end{abstract}



\section{Introduction}\label{sec:introduction}

Throughout the paper, $k\geq2$ is an integer, and all graphs are finite
and simple. A \emph{$\chi_k'$-coloring} of a graph $G$ is an edge-coloring in
which the subgraph spanned by the edges of each color has all nonzero
degrees congruent to $1\pmod{k}$. The \emph{mod $k$ chromatic index}
$\chi_k'(G)$ is the minimum number of colors in such a coloring. A graph
$G$ is a \emph{$0_k$-graph} if $d_G(v)\equiv0\pmod{k}$ for every
$v\in V(G)$.

The case $k=2$ is the odd edge-coloring problem. In this setting, the
edges of each color span an odd subgraph, meaning that every nonisolated
vertex of the corresponding color subgraph has odd degree.
Pyber~\cite{Pyber1992} initiated the study of the odd chromatic index and
proved that every simple graph has odd chromatic index at most four. Odd
edge-colorings have since been studied for several graph classes and for
loopless multigraphs; see, for example,
\cite{LuzarPetrusevskiSkrekovski2015,AtanasovPetrusevskiSkrekovski2016,
Petrusevski2018}. Related decomposition problems for odd subgraphs and
matchings were investigated by Kano and Katona~\cite{KanoKatona2002},
and M\'atrai~\cite{Matrai2006} studied coverings by three odd subgraphs.
Thus the definition above extends the parity condition from modulus two
to arbitrary modulus.

Pyber~\cite{Pyber1992} asked whether $\chik(G)$ is bounded by a function
of $k$ alone. Scott~\cite{Scott1997} answered this question affirmatively
by showing that $\chik(G)\leq5k^2\log k$ for every graph $G$, and asked
whether a linear bound in $k$ always holds. Botler, Colucci, and
Kohayakawa~\cite{BotlerColucciKohayakawa2023} proved
$\chik(G)\leq198k-101$ and proposed the following stronger additive
form, stated as Conjecture~1 in their paper.

\begin{conjecture}[Botler, Colucci, and Kohayakawa~\cite{BotlerColucciKohayakawa2023}]
\label{conj:additive-constant}
There exists an absolute constant $C$ such that
$\chik(G)\leq k+C$ for every integer $k\geq2$ and every graph $G$.
\end{conjecture}

Their work on random graphs also showed that the mod $k$ chromatic index
is close to its natural lower bound $k$ with high probability over a
wide range of edge probabilities~\cite{BotlerColucciKohayakawaRandom2023}.
Nweit and Yang~\cite{NweitYang2024} subsequently improved the general
linear bound to $177k-93$. Berthe et al.~\cite{BertheEtAl2026} later proved
$\chik(G)\leq9k+o(k)$ for all $k$ and $\chik(G)\leq7k+o(k)$ when $k$ is
odd. Their method, as in the earlier linear-bound arguments, is closely
related to divisible subgraphs and graph factors with prescribed
degrees modulo $k$; see also
\cite{AlonFriedlandKalai1984,Mader1972,Thomassen2014}. In their concluding
remarks, they proposed the following conjecture, which is Conjecture~13
in their paper.

\begin{conjecture}[Berthe et al.~\cite{BertheEtAl2026}]\label{conj:berthe}
There exists a function $f(k)=o(k)$ such that every $0_k$-graph $G$
satisfies $\chik(G)\leq k+f(k)$.
\end{conjecture}

The conjecture is natural because, in any $\chi_k'$-coloring of a
$0_k$-graph, each color appearing at a vertex contributes $1$ modulo
$k$ to its degree. Hence, at every nonisolated vertex, the number of
colors that appear is a positive multiple of $k$ and is therefore at
least $k$. The conjecture asks whether the whole graph can always be
colored using only slightly more than this local lower bound.

Very recently, Liu, Xu, and Yang~\cite{LiuXuYang2026} disproved
Conjecture~\ref{conj:additive-constant}, even for bipartite graphs.
For all sufficiently large $k$, they constructed bipartite graphs $G$
with
\[
\chi_k'(G)\geq \frac32 k-O((k\log k)^{1/3}).
\]
Their constructions are not $0_k$-graphs, and the codegree obstruction
used in their proofs does not directly apply under the requirement that
every vertex degree be divisible by $k$. Thus their result does not
settle Conjecture~\ref{conj:berthe}.

We show that Conjecture~\ref{conj:berthe} is false even under the
stronger restriction that every vertex degree belongs to $\{k,2k\}$.
Our main result is the following.

\begin{theorem}\label{thm1}
There is a sequence of integers $t_k$, with $1\leq t_k\leq k-1$,
such that the following holds. Let $\{G_k\}_{k\geq2}$ be a sequence
of graphs, where $G_k$ has a vertex partition
$V(G_k)=X_k\cup Y_k$ with $|X_k|=2k-t_k$ and $|Y_k|=2k$.
Suppose that $A_k\subseteq X_k$ has size $t_k$, every vertex of
$A_k$ has degree $2k$, every vertex of $V(G_k)\setminus A_k$ has
degree $k$, and $e(G_k[X_k])=o(k^2)$. Then
\[
\chik(G_k)\geq\left(4-2\sqrt2+o(1)\right)k.
\]
\end{theorem}

\begin{corollary}\label{cor:counterexamples}
Both Conjectures~\ref{conj:additive-constant} and~\ref{conj:berthe} are
false even for connected bipartite $0_k$-graphs with degree set
$\{k,2k\}$. They are also false for connected nonbipartite $0_k$-graphs
with the same degree set.
\end{corollary}

\section{Preliminaries}\label{sec:notation}

We use standard graph-theoretic notation. For a graph $G$, let $d_G(v)$
denote the degree of $v$, and let $G[S]$ be the subgraph induced by
$S\subseteq V(G)$. For two disjoint vertex sets $S$ and $T$, we say
that $S$ is \emph{complete to} $T$ if every vertex of $S$ is adjacent
to every vertex of $T$. A vertex of degree $j$ is called a $j$-vertex,
and $[q]=\{1,2,\ldots,q\}$.

Let $\varphi:E(G)\to[q]$ be a $\chi_k'$-coloring of $G$. For
$c\in[q]$, let $E_c=\{e\in E(G):\varphi(e)=c\}$ and
$G_c=(V(G),E_c)$. Thus $G_c$ is the spanning subgraph formed by the
edges of color $c$. We call $d_{G_c}(v)$ the \emph{color degree} of
$c$ at $v$, and say that $c$ \emph{appears at} $v$ if
$d_{G_c}(v)>0$. Every positive color degree is congruent to
$1\pmod{k}$.

\begin{lemma}\label{lem:k-colors-at-vertex}
Let $G$ be a $0_k$-graph with no isolated vertices, and suppose that
$G$ has a $\chi_k'$-coloring with fewer than $2k$ colors. Then exactly
$k$ colors appear at every vertex. Moreover, at a $k$-vertex $v$, every
color appearing at $v$ has color degree $1$. At a $2k$-vertex $v$, one
color appearing at $v$ has color degree $k+1$, and each of the other
$k-1$ colors appearing at $v$ has color degree $1$.
\end{lemma}

\begin{proof}
For a vertex $v$, let $r(v)$ be the number of colors that appear at
$v$. Each positive color degree is congruent to $1$ modulo $k$, while
$d_G(v)\equiv0\pmod{k}$. Hence $r(v)\equiv0\pmod{k}$. Since $v$ is not
isolated and fewer than $2k$ colors are used, $1\leq r(v)<2k$, so
$r(v)=k$.

If $d_G(v)=k$, the $k$ positive color degrees sum to $k$, so all are
$1$. If $d_G(v)=2k$, each positive color degree is congruent to $1$
modulo $k$ and is at most $2k$, so it is either $1$ or $k+1$. Starting
with $k$ contributions of $1$ leaves a total of $k$, so exactly one
color has color degree $k+1$ and all the others have color degree $1$.
\end{proof}

Whenever a $\chi_k'$-coloring with fewer than $2k$ colors is fixed, the
unique color having color degree $k+1$ at a $2k$-vertex is called the
\emph{heavy color} at that vertex.

\section{A lower bound}\label{sec:lower-bound}

Let $1\leq t\leq k-1$. We consider graphs with a vertex partition
$X\cup Y$ of sizes $2k-t$ and $2k$. The sets $X$ and $Y$ need not be
independent.

\begin{lemma}\label{lem:general-lower-bound}
Assume that $G$ has a vertex partition $V(G)=X\cup Y$ with $|X|=2k-t$ and
$|Y|=2k$. Suppose that $A\subseteq X$ has size $t$, every vertex of
$A$ has degree $2k$, and every vertex of $V(G)\setminus A$ has degree
$k$. Set $p_X=e(G[X])$. Then
\begin{equation}\label{eq:general-bound}
\chik(G)\geq
k+\max\left\{0,
\left\lceil\frac{t(k-t)-2p_X}{2k-t}\right\rceil
\right\}.
\end{equation}
\end{lemma}

\begin{proof}
Since $p_X\geq0$ and $t(k-t)/(2k-t)<t$, the term added to $k$ in
\eqref{eq:general-bound} is at most $t\leq k-1$. Thus there is nothing
to prove if $\chik(G)\geq2k$.

Suppose that $\chik(G)<2k$. Let $q=\chik(G)$, fix a
$\chi_k'$-coloring with color set $[q]$, and let $q=k+s$. The degree
assumptions make $G$ a $0_k$-graph with no isolated vertices, so
Lemma~\ref{lem:k-colors-at-vertex} shows that exactly $k$ colors appear
at every vertex and $s\geq0$.

For $c\in[q]$, let $N_c^X$ and $N_c^Y$ be the numbers of vertices in
$X$ and $Y$, respectively, at which $c$ appears. Let $\mu_c$ be the
number of vertices in $A$ whose heavy color is $c$, and put
$e_c^X=e(G_c[X])$ and $e_c^Y=e(G_c[Y])$. Comparing the two degree sums
for color $c$ gives
\[
N_c^X+k\mu_c-N_c^Y=2(e_c^X-e_c^Y).
\]
Indeed, each appearance in $X\setminus A$ or $Y$ contributes $1$, while
a heavy color at a vertex of $A$ contributes an additional $k$. Edges
between $X$ and $Y$ contribute once to each degree sum and cancel,
whereas edges inside $X$ and $Y$ contribute twice.

Let $\mathcal C_h=\{c\in[q]:\mu_c>0\}$ and
$q_h=|\mathcal C_h|$. Since every vertex of $A$ has exactly one heavy
color, $\sum_{c\in\mathcal C_h}\mu_c=t$ and $q_h\leq t$. If
$c\notin\mathcal C_h$, then $N_c^X\leq2k-t$. If
$c\in\mathcal C_h$, the preceding identity, together with
$N_c^Y\leq2k$ and $e_c^Y\geq0$, gives
$N_c^X\leq2k-k\mu_c+2e_c^X$.

Counting the pairs $(v,c)$ with $v\in X$ at which $c$ appears gives
\begin{align*}
k(2k-t)=\sum_{c=1}^{q}N_c^X
&\leq (q-q_h)(2k-t)
 +\sum_{c\in\mathcal C_h}(2k-k\mu_c+2e_c^X)\\
&=q(2k-t)+q_ht-kt
 +2\sum_{c\in\mathcal C_h}e_c^X\\
&\leq q(2k-t)-t(k-t)+2p_X.
\end{align*}
The last inequality uses $q_h\leq t$ and
$\sum_{c\in\mathcal C_h}e_c^X\leq p_X$. Substituting $q=k+s$ gives
$s(2k-t)\geq t(k-t)-2p_X$. Since $s$ is a nonnegative integer,
\eqref{eq:general-bound} follows.
\end{proof}

The degree sums over $X$ and $Y$ are both $2k^2$. If $m$ is the number
of edges between $X$ and $Y$, then
$2k^2=m+2e(G[X])=m+2e(G[Y])$. Hence
$e(G[X])=e(G[Y])$. Thus the same number of internal edges occurs on the
two sides of the partition.

\subsection{Choosing \texorpdfstring{$t$}{t}}
Let $g_k(t)=t(k-t)/(2k-t)$ for $1\leq t\leq k-1$.

\begin{lemma}\label{lem:choice-t}
For $k\geq2$, let
$\theta_k=(4k-1-\sqrt{8k^2+1})/2$ and
$t_k=\lceil\theta_k\rceil$. Then $1\leq t_k\leq k-1$ and
$g_k(t_k)=\max_{1\leq t\leq k-1}g_k(t)$. Moreover,
\[
\frac{t_k}{k}\longrightarrow2-\sqrt2
\qquad\text{and}\qquad
\frac{g_k(t_k)}{k}\longrightarrow3-2\sqrt2.
\]
\end{lemma}

\begin{proof}
For $k=2$, the only admissible value of $t$ is $1$, and
$0<\theta_2<1$, so $t_2=1$ and the assertion is immediate.
Hence assume $k\geq3$. For
$1\leq t\leq k-2$, direct calculation gives
\[
g_k(t+1)-g_k(t)=
\frac{t^2+(1-4k)t+2k^2-2k}{(2k-t)(2k-t-1)}.
\]
The denominator is positive. The roots of the numerator are $\theta_k$
and $(4k-1+\sqrt{8k^2+1})/2$, and the second root is larger than
$k-1$. Hence the difference is positive for integers $t<\theta_k$ and
negative for integers $t>\theta_k$; if $\theta_k$ is an integer, it is
zero there. This proves the assertion about $g_k(t_k)$.

Since $(4k-1)^2-(8k^2+1)=8k(k-1)>0$ and
$8k^2+1-(2k+1)^2=4k(k-1)>0$, we have $0<\theta_k<k-1$. Finally,
$\theta_k/k\to2-\sqrt2$, so $t_k/k\to2-\sqrt2$. Since
$g_k(t)/k=(t/k)(1-t/k)/(2-t/k)$, we obtain
$g_k(t_k)/k\to3-2\sqrt2$.
\end{proof}

\begin{proof}[Proof of Theorem~\ref{thm1}]
Apply Lemma~\ref{lem:general-lower-bound} with $t=t_k$ and
$p_X=e(G_k[X_k])$. By Lemma~\ref{lem:choice-t}, we have 
$g_k(t_k)=(3-2\sqrt2+o(1))k$. Also,
$2p_X/(2k-t_k)=o(k)$ because $p_X=o(k^2)$ and
$t_k/k\to2-\sqrt2$. Thus the quantity inside the ceiling in
\eqref{eq:general-bound} is positive for all sufficiently large $k$, and
\[
\chik(G_k)\geq(4-2\sqrt2+o(1))k.
\]
\end{proof}

Since $p_X\geq0$, the right-hand side of \eqref{eq:general-bound} is
largest when $p_X=0$. Lemma~\ref{lem:choice-t} then shows that no choice
of $t$ gives a larger first-order coefficient than $4-2\sqrt2$ in this
lemma. This does not rule out a larger lower bound for other degree
patterns or by a different argument.

\section{Examples}\label{sec:examples}

The examples are described for an arbitrary $t$ with
$1\leq t\leq k-1$. Taking $t=t_k$ then gives the bound in
Theorem~\ref{thm1}. We consider both bipartite and nonbipartite
examples. The first figure only illustrates one bipartite graph; its
particular arrangement is not used in the lower bound.

\subsection{Bipartite examples}

For $1\leq t\leq k-1$, let $\mathcal B_{k,t}$ be the class of bipartite
graphs $G$ with partite sets $X$ and $Y$ such that $|X|=2k-t$,
$|Y|=2k$, exactly $t$ vertices of $X$ have degree $2k$, and every other
vertex has degree $k$. Let $A\subseteq X$ be the set of $2k$-vertices
and put $B=X\setminus A$. Since $|Y|=2k$, every vertex of $A$ is
complete to $Y$.

For example, split $B$ into sets $B_1$ and $B_2$ of size $k-t$, split
$Y$ into sets $Y_1$ and $Y_2$ of size $k$, make $A$ complete to $Y$,
and make $B_i$ complete to $Y_i$ for $i\in\{1,2\}$. No other edges are
added. Figure~\ref{fig:block-family} shows this graph. Graphs in
$\mathcal B_{k,t}$ may not have this form.

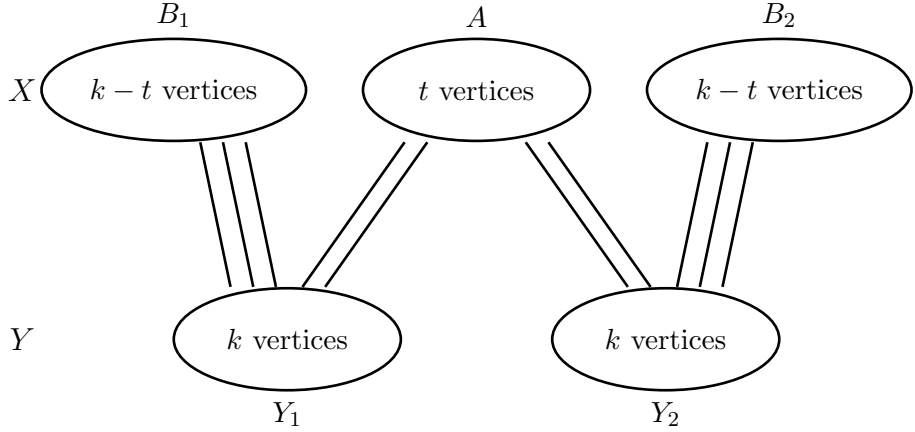
\begin{figure}[htbp]
\centering
\begin{tikzpicture}[
    x=1cm,y=1cm,
    set/.style={ellipse,draw=black,line width=1pt,minimum width=3.0cm,
                minimum height=1.35cm,align=center},
    edge/.style={draw=black,line width=1pt},
    lab/.style={font=\large}
]

\node[lab] at (-6.0,1.8) {$X$};
\node[lab] at (-6.0,-1.5) {$Y$};

\node[set,label=above:$B_1$] (B1) at (-4.0,1.8)
      {$k-t$ vertices};
\node[set,label=above:$A$] (A) at (0,1.8)
      {$t$ vertices};
\node[set,label=above:$B_2$] (B2) at (4.0,1.8)
      {$k-t$ vertices};

\node[set,label=below:$Y_1$] (Y1) at (-2.5,-1.5)
      {$k$ vertices};
\node[set,label=below:$Y_2$] (Y2) at (2.5,-1.5)
      {$k$ vertices};

\draw[edge] ($(B1.south)+(0.35,0)$) -- ($(Y1.north)+(-0.75,0)$);
\draw[edge] ($(B1.south)+(0.65,0)$) -- ($(Y1.north)+(-0.45,0)$);
\draw[edge] ($(B1.south)+(0.95,0)$) -- ($(Y1.north)+(-0.15,0)$);

\draw[edge] ($(A.south)+(-0.95,0)$) -- ($(Y1.north)+(0.20,0)$);
\draw[edge] ($(A.south)+(-0.65,0)$) -- ($(Y1.north)+(0.50,0)$);
\draw[edge] ($(A.south)+(0.65,0)$) -- ($(Y2.north)+(-0.50,0)$);
\draw[edge] ($(A.south)+(0.95,0)$) -- ($(Y2.north)+(-0.20,0)$);

\draw[edge] ($(B2.south)+(-0.95,0)$) -- ($(Y2.north)+(0.15,0)$);
\draw[edge] ($(B2.south)+(-0.65,0)$) -- ($(Y2.north)+(0.45,0)$);
\draw[edge] ($(B2.south)+(-0.35,0)$) -- ($(Y2.north)+(0.75,0)$);

\end{tikzpicture}
\caption{A simple member of $\mathcal B_{k,t}$. Each group of lines
represents all edges between the corresponding sets.}
\label{fig:block-family}
\end{figure}

In the graph described above, every vertex of $A$ has degree $2k$,
every vertex of $B_i$ has degree $k$, and every vertex of $Y_i$ has
degree $t+(k-t)=k$. Hence $\mathcal B_{k,t}$ is nonempty. Moreover,
every graph in $\mathcal B_{k,t}$ is connected. Indeed, choose
$a\in A$. Since $N_G(a)=Y$, all vertices of $Y$ lie in one component.
Every vertex of $B$ has positive degree and hence has a neighbor in
$Y$. Thus every graph in $\mathcal B_{k,t}$ is a connected bipartite
$0_k$-graph with degree set $\{k,2k\}$.

For every choice $G_k\in\mathcal B_{k,t_k}$, we have
$e(G_k[X_k])=0$. Hence Theorem~\ref{thm1} gives
\[
\chik(G_k)\geq(4-2\sqrt2+o(1))k.
\]
Thus the bipartite graphs alone already disprove both conjectures.

\subsection{Nonbipartite examples}

The same bound also holds for connected nonbipartite graphs. Fix
$G\in\mathcal B_{k,t}$. Choose $a\in A$ and $b\in B$. Since
$d_G(b)=k\geq2$, choose distinct vertices $v,w\in N_G(b)$, and choose
$u\in Y\setminus\{v,w\}$.

The edges $au$ and $bv$ are present. Delete them and add $ab$ and $uv$.
The two new edges were not present because their ends lie in the same
parts of the bipartition. Let $F$ be the resulting graph. The edges
$aw$ and $bw$ are unchanged, so $a,b,w$ form a triangle in $F$.
Figure~\ref{fig:edge-change} shows the four changed edges.

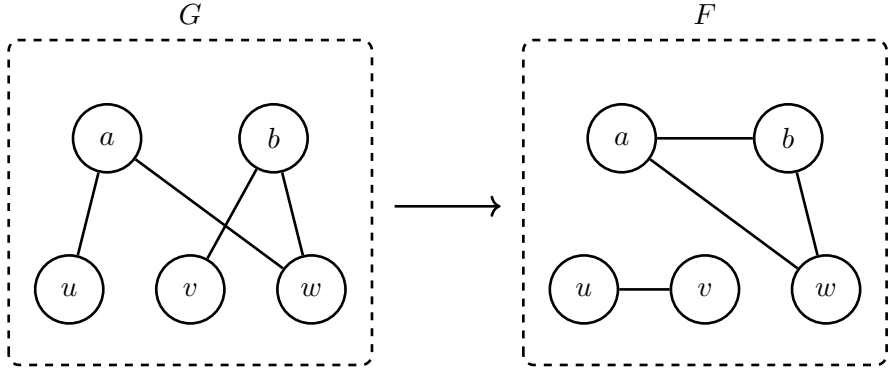
\begin{figure}[htbp]
\centering
\begin{tikzpicture}[
    scale=1,
    every node/.style={font=\normalsize},
    vtx/.style={circle,draw=black,line width=1pt,minimum size=9mm,inner sep=0pt},
    changed/.style={draw=black,line width=1pt},
    kept/.style={draw=black,line width=1pt},
    arr/.style={->,line width=1pt}
]

\draw[dashed,rounded corners,line width=1pt] (-5.8,-2.0) rectangle (-1.0,2.3);
\draw[dashed,rounded corners,line width=1pt] (1.0,-2.0) rectangle (5.8,2.3);

\node at (-3.4,2.65) {$G$};
\node at (3.4,2.65) {$F$};

\node[vtx] (al) at (-4.5,1.0) {$a$};
\node[vtx] (bl) at (-2.3,1.0) {$b$};
\node[vtx] (ul) at (-5.0,-1.0) {$u$};
\node[vtx] (vl) at (-3.4,-1.0) {$v$};
\node[vtx] (wl) at (-1.8,-1.0) {$w$};

\draw[changed] (al) -- (ul);
\draw[changed] (bl) -- (vl);
\draw[kept] (al) -- (wl);
\draw[kept] (bl) -- (wl);

\draw[arr] (-0.7,0.1) -- (0.7,0.1);

\node[vtx] (ar) at (2.3,1.0) {$a$};
\node[vtx] (br) at (4.5,1.0) {$b$};
\node[vtx] (ur) at (1.8,-1.0) {$u$};
\node[vtx] (vr) at (3.4,-1.0) {$v$};
\node[vtx] (wr) at (5.0,-1.0) {$w$};

\draw[changed] (ar) -- (br);
\draw[changed] (ur) -- (vr);
\draw[kept] (ar) -- (wr);
\draw[kept] (br) -- (wr);

\end{tikzpicture}
\caption{The graph $F$ is obtained from $G$ by deleting $au$ and $bv$
and adding $ab$ and $uv$. The edges $aw$ and $bw$ are unchanged,
so $a,b,w$ form a triangle in $F$. Other edges are omitted.}
\label{fig:edge-change}
\end{figure}

At each of $a$, $b$, $u$, and $v$, one incident edge is deleted and
one is added, so all vertex degrees remain unchanged. The new edges
$ab$ and $uv$ were absent from $G$, and hence $F$ is simple. Moreover,
$e(F[X])=e(F[Y])=1$.

Every vertex of $Y\setminus\{u\}$ is still adjacent to $a$. The vertex
$u$ is joined to $v$ by $uv$, and $v$ is adjacent to $a$. The vertex
$b$ is adjacent to $a$ through $ab$. Every vertex of
$B\setminus\{b\}$ has $k\geq2$ neighbors in $Y$, so it has a neighbor
different from $u$, while every vertex of $A\setminus\{a\}$ remains
adjacent to all of $Y$. Hence $F$ is connected. Since $a,b,w$ form a
triangle, $F$ is nonbipartite. Thus $F$ is a connected nonbipartite
$0_k$-graph with degree set $\{k,2k\}$.

For each $k$, choose $G_k\in\mathcal B_{k,t_k}$ and let $F_k$ be
obtained from $G_k$ as above. Then
$e(F_k[X_k])=1=o(k^2)$, and Theorem~\ref{thm1} gives
\[
\chik(F_k)\geq(4-2\sqrt2+o(1))k.
\]

\begin{proof}[Proof of Corollary~\ref{cor:counterexamples}]
The bipartite examples above are connected. They are $0_k$-graphs with
degree set $\{k,2k\}$ and satisfy
\[
\chik(G_k)-k\geq(3-2\sqrt2+o(1))k.
\]
Since $3-2\sqrt2>0$, this difference is neither bounded by an absolute
constant nor $o(k)$, so both conjectures are false. The graphs $F_k$
give the same conclusion in the nonbipartite case.
\end{proof}

\section{Further questions}\label{sec:consequences}

The coefficient $4-2\sqrt2$ is the largest one obtained from
Lemma~\ref{lem:general-lower-bound}. It is not clear whether it is
best possible for $0_k$-graphs with degree set $\{k,2k\}$.

\begin{problem}\label{prob:degree-set}
Can the coefficient $4-2\sqrt2$ be improved for $0_k$-graphs with
degree set $\{k,2k\}$?
\end{problem}

We do not know whether the bound in
Lemma~\ref{lem:general-lower-bound} itself is sharp.

\begin{problem}\label{prob:lemma-sharp}
Is the bound in Lemma~\ref{lem:general-lower-bound} sharp?
\end{problem}

\vspace{2mm}\noindent\textbf{Declarations}

\vspace{2mm}\noindent\textbf{Competing interests.}
The authors declare that they have no competing interests.

\vspace{2mm}\noindent\textbf{Data availability statement.}
This manuscript has no associated data.

\phantomsection
\label{sec:references}


\begin{thebibliography}{99}

\bibitem{AlonFriedlandKalai1984}
N. Alon, S. Friedland, G. Kalai,
Regular subgraphs of almost regular graphs,
J. Combin. Theory Ser. B 37 (1984) 79--91.

\bibitem{AtanasovPetrusevskiSkrekovski2016}
R. Atanasov, M. Petru\v{s}evski, R. \v{S}krekovski,
Odd edge-colorability of subcubic graphs,
Ars Math. Contemp. 10 (2016) 359--370.

\bibitem{BertheEtAl2026}
G. Berthe, M. Bonamy, F. Botler, G. Carenini, L. Colucci, A. Dumas,
F. Ghasemi, P. M. Viana Neto,
On modular edge colorings of graphs,
SIAM J. Discrete Math. 40 (2026) 897--904.

\bibitem{BotlerColucciKohayakawa2023}
F. Botler, L. Colucci, Y. Kohayakawa,
The mod $k$ chromatic index of graphs is $O(k)$,
J. Graph Theory 102 (2023) 197--200.

\bibitem{BotlerColucciKohayakawaRandom2023}
F. Botler, L. Colucci, Y. Kohayakawa,
The mod $k$ chromatic index of random graphs,
J. Graph Theory 103 (2023) 767--779.

\bibitem{KanoKatona2002}
M. Kano, G. Y. Katona,
Odd subgraphs and matchings,
Discrete Math. 250 (2002) 265--272.

\bibitem{LiuXuYang2026}
X.-C. Liu, B. Xu, X. Yang,
Linear lower bounds for the modular chromatic index,
arXiv:2608.02239 (2026).

\bibitem{LuzarPetrusevskiSkrekovski2015}
B. Lu\v{z}ar, M. Petru\v{s}evski, R. \v{S}krekovski,
Odd edge coloring of graphs,
Ars Math. Contemp. 9 (2015) 277--287.

\bibitem{Mader1972}
W. Mader,
Existenz $n$-fach zusammenh\"an\-gender Teil\-graphen in Graphen gen\"u\-gend
grosser Kanten\-dichte,
Abh. Math. Semin. Univ. Hambg. 37 (1972) 86--97.

\bibitem{Matrai2006}
T. M\'atrai,
Covering the edges of a graph by three odd subgraphs,
J. Graph Theory 53 (2006) 75--82.

\bibitem{NweitYang2024}
O. Nweit, D. Yang,
On the mod $k$ chromatic index of graphs,
Discrete Math. Theor. Comput. Sci., 26 (2024), no. 3,
Art. 16, 6 pp.

\bibitem{Petrusevski2018}
M. Petru\v{s}evski,
Odd $4$-edge-colorability of graphs,
J. Graph Theory 87 (2018) 460--474.

\bibitem{Pyber1992}
L. Pyber,
Covering the edges of a graph by $\ldots$,
in Sets, Graphs and Numbers (Budapest, 1991), Colloq. Math. Soc. J\'anos
Bolyai 60, North-Holland, Amsterdam (1992) 583--610.

\bibitem{Scott1997}
A. D. Scott,
On graph decompositions modulo $k$,
Discrete Math. 175 (1997) 289--291.

\bibitem{Thomassen2014}
C. Thomassen,
Graph factors modulo $k$,
J. Combin. Theory Ser. B 106 (2014) 174--177.

\end{thebibliography}
\end{document}